\documentclass{amsart}

\usepackage{ae}
\usepackage[all]{xy}
\usepackage{graphicx,amsfonts,amssymb,amsmath}
\usepackage{natbib}

\usepackage[latin1]{inputenc}
\usepackage[T1]{fontenc}
\usepackage[english]{babel}
\usepackage[cyr]{aeguill}

\newcommand{\chapeau}{{\rlap{\smash{\hbox{\lower4pt\hbox{\hskip1pt$\widehat{\phantom{u}}$}}}}}\mbox{ }}
\usepackage[OT2,T1]{fontenc}
\DeclareSymbolFont{cyrletters}{OT2}{wncyr}{m}{n}
\DeclareMathSymbol{\sha}{\mathalpha}{cyrletters}{"58}

 \newtheorem{thm}{Theorem}[section]

 \newtheorem*{pb*}{\textit{Open question}}
 \newtheorem*{cor*}{\textit{Corollary}}

 \newtheorem{cor}[thm]{Corollary}
 \newtheorem{lem}[thm]{Lemma}
 \newtheorem{prop}[thm]{Proposition}
 \theoremstyle{definition}
 
 \theoremstyle{remark}
 
 \theoremstyle{remark}
 \newtheorem{rem}[thm]{Remark}
 \theoremstyle{remark}
 \newtheorem{step}{Step}[thm]
 \newtheorem{substep}{Step}[step]

 \numberwithin{equation}{subsection}

 \newcommand{\To}{\longrightarrow}

 \renewcommand{\P}{\mathbb{P}}
 
 \newcommand{\Z}{\mathbb{Z}}

 \newcommand{\Gal}{\textup{Gal}}
 \newcommand{\coker}{\textup{Coker}}
 
 \newcommand{\Br}{\textup{Br}}
 \newcommand{\Pic}{\textup{Pic}}
 \newcommand{\CH}{\textup{CH}}
 \newcommand{\Hom}{\textup{Hom}}
 
 \newcommand{\inv}{\textup{inv}}
 \newcommand{\E}{\textup{E}}

\usepackage{hyperref}

\begin{document}

\title[Brauer\textendash Manin obstruction on fibred varieties]
{Towards the Brauer\textendash Manin obstruction on varieties fibred over the projective line}

\author{ Yongqi LIANG  }

\address{Yongqi LIANG
\newline B\^atiment Sophie Germain,
\newline Universit\'e Paris Diderot - Paris 7,
\newline Institut de Math\'ematiques de Jussieu,
 \newline  75013 Paris,\newline
 France}

\email{liangy@math.jussieu.fr}



\date{\today.}



\maketitle

\begin{abstract}
Recently Dasheng Wei proved that the Brauer\textendash Manin
obstruction is the only obstruction to the Hasse principle for 0-cycles of degree $1$ on
some fibrations over the projective line defined by bi-cyclic normic equations. In the present paper, we prove the exactness of
the global-to-local sequence for Chow groups of 0-cycles of such varieties, which signifies that the Brauer\textendash Manin
obstruction is also the only obstruction to weak approximation for 0-cycles of arbitrary degree.
Our main theorem also generalizes several existing results.
\end{abstract}

\tableofcontents

\section{Introduction}

Let $k$ be a number field and $\Omega_k$ be the set of all places of $k.$
Consider an algebraic $k$-variety $X,$ proper smooth and geometrically integral over $k$,
denote by $\Br(X)=\textup{H}_{\mbox{\scriptsize{\'et}}}^2(X,\mathbb{G}_m)$
the cohomological Brauer group of $X.$ We write simply $X_v=X\times_kk_v$ for each place $v\in\Omega_k.$
One defines the \emph{modified Chow group} of 0-cycles $\CH'_0(X_v)$ as in \cite{Wittenberg}
to be either the usual Chow group if $v$ is a non-archimedean place,
or $\coker[\CH_0(X_v\times_{k_v}\bar{k}_v)\buildrel{N_{\bar{k}_v|k_v}}\over\To \CH_0(X_v)]$ if $v$ is an archimedean place.

As a variant of Manin's pairing, one can define by taking the sum of local evaluations
$$\Br(X)\times \prod_{v\in\Omega_k}\CH'_0(X_v)\to\mathbb{Q}/\mathbb{Z}$$
$$(\mbox{ }b\mbox{ },\{z_v\}_v)\mapsto\sum_{v\in\Omega_k}\inv_v(\langle b,z_v\rangle_v),$$
where $\inv_v:\Br(k_v)\hookrightarrow\mathbb{Q}/\mathbb{Z}$ is the local invariant, \emph{cf.}
\cite{CT95}.
By class field theory, global 0-cycles are annihilated under the pairing,
we get a complex
$$\varprojlim_n\CH_0(X)/n\to\prod_{v\in\Omega_k}\varprojlim_n\CH'_0(X_v)/n\to \Hom(\Br(X),\mathbb{Q}/\mathbb{Z}), \leqno(\E)$$
where $/n$ denote the cokernel of the multiplication by $n.$
The exactness of $(E)$ means roughly that the failure of the Hasse principle and weak approximation for 0-cycles can be controlled
by the Brauer group $\Br(X).$
Summarized by Wittenberg \cite{Wittenberg}, the sequence is conjectured to be exact for all smooth proper geometrically integral varieties after the work of Colliot-Th\'el\`ene,
Kato, Sansuc, Saito, \emph{cf.} \cite{CTSansuc81} \cite{KatoSaito86} \cite{CT95}.

The conjecture was initially studied for conic bundle surfaces over the projective line,
\cite{CTSansuc81} \cite{chateletsurfaces0} \cite{chateletsurfaces} \cite{Salberger}. General results were obtained for varieties fibred
over the projective line  assuming various hypotheses, \cite{CT-SD} \cite{CT-Sk-SD} \cite{WittenbergLNM} \cite{Liang3} \cite{Liang2}.
In these papers, either all the fibres are supposed to be split or the Hasse principle/weak approximation is supposed to be satisfied
on most fibres.
Until recently, Wei considered fibrations defined by bi-cyclic  normic equations, on which neither of the two above assumptions is satisfied.
He proved that the Brauer\textendash Manin obstruction is the only obstruction to the Hasse principle for 0-cycles of degree $1$ on
such fibrations \cite[Thm. 4.1]{Wei}, sophisticated calculation of Brauer groups was involved in his proof.
However, Wei's discussion concerns only the Hasse principle for 0-cycles of degree $1$ and his method does not extend directly to
0-cycles of arbitrary degree.
In this paper, we will prove for such fibrations the exactness of the sequence $(\E)$ as stated in Corollary \ref{maincor},
which concerns the weak approximation property for 0-cycles of arbitrary degree.
It follows from our main result Theorem \ref{main_thm} after a remark of Smeets.
Actually, Theorem \ref{main_thm} generalizes several existing results at the same time, and it also gives an alternative proof of Wei's result,
\textit{cf.} \S \ref{application}.

A detailed proof is given in \S \ref{proof} after some preliminaries presented in \S \ref{preliminaries}.
We apply the fibration method and we make use of not only the vertical Brauer group (to show local solvability)
but also non-constant elements of the Brauer group of the generic fibre (to show orthogonality to the Brauer
group of a selected fibre). We need the tool of generalized Hilbertian subsets to deal with difficulties
caused by degrees of 0-cycles.
In the last section \S \ref{application}, we discuss explicit examples.

\section{Statement of the main result}\label{main_result}

Let $X$ be a smooth proper geometrically integral variety defined over a number field $k.$ Fix an integer $\delta.$
If the existence of a family $\{z_v\}_{v\in\Omega_k}$ of local 0-cycles
of degree $\delta$ orthogonal to $\Br(X)$ implies the existence of a global 0-cycle of degree $\delta$ on $X,$
then we say that \emph{the Brauer\textendash Manin obstruction is the only obstruction to the Hasse principle}
for 0-cycles of degree $\delta$ on $X.$ We say that
\emph{the Brauer\textendash Manin obstruction is the only obstruction to weak approximation}
for 0-cycles of degree $\delta$ on $X$ if the following statement is satisfied
\begin{itemize}
  \item[] For any positive integer $N$ and any finite set $S$ of places of $k,$ given an arbitrary family of local 0-cycles
            $\{z_v\}_{v\in\Omega_{k(\theta)}}$ of degree $\delta$ orthogonal to $\Br(X),$ then there exists a global
            0-cycle $z=z_{S,N}$ of degree $\delta$ on $X$ such that $z$ and $z_v$ have the same image in
            $\CH_0(X_v)/N$ for all $v\in S.$
\end{itemize}
These definitions date back to \cite[page 69]{CT-SD}.

Let $\pi:X \to \mathbb{P}^1$ be a fibration (dominant morphism whose generic fibre is geometrically integral)
defined over a number field $k.$
Suppose that the variety $X$ is smooth projective and geometrically integral.
We denote by $X_\eta$ the generic fibre of $\pi,$ and we write
$X_{\bar{\eta}}=X_\eta\times_{k(t)}\overline{k(t)}.$
Suppose that $\Pic(X_{\bar{\eta}})$
is torsion-free and $\Br(X_{\bar{\eta}})$ is finite, then
$\Br(X_\eta)/\Br(k(t))$ is a finite group. Let $\Lambda\subset \Br(X_\eta)$ be
a finite subset generating $\Br(X_\eta)$ modulo $\Br(k(t)).$
Let $F$ be an integral $1$-codimensional closed subvariety of $X.$
Consider the residue map $\partial_F:\Br(k(X))\to H^1(k(F),\mathbb{Q}/\mathbb{Z}),$
if there exists an element $b\in \Lambda\subset \Br(X_\eta)$ such that $\partial_F(b)\neq0$
then $F$ is contained in a certain closed fibre $X_m$ of $\pi.$
Let $\{m_r;1\leqslant r\leqslant l\}$ be a set of closed points of $\mathbb{P}^1.$
We denote by $X_r$ the fibre $X_{m_r}$ and we write $Z=\bigsqcup_rX_r$ and $Y=X\setminus Z.$

\begin{thm}\label{main_thm}
With the above notation,
let $\pi:X \to \mathbb{P}^1$ be a fibration satisfying
\begin{itemize}
\item[(ab-sp)] all closed fibres are abelian-split;
\item[(gen)] $\Pic(X_{\bar{\eta}})$ is torsion-free and $\Br(X_{\bar{\eta}})$ is finite;
\item[$(\Br)$] there exists a finite subset $\Lambda\subset \Br(X_\eta)$ generating $\Br(X_\eta)$ modulo
$\Br(k(t))$ such that $\Lambda\subset \Br(Y)$ where $Y\subset X$ is an open with complement $X\setminus Y$
a disjoint union of closed split fibres.
\end{itemize}
Let $\textsf{Hil}\subset \mathbb{P}^1$ be a generalized Hilbertian subset of $\mathbb{P}^1,$ we suppose respectively that
for all $\theta\in \textsf{Hil}$
\begin{itemize}
\item[(1)] the Brauer\textendash Manin obstruction is the only obstruction to the Hasse principle for rational points (or 0-cycles of degree $1$)
on $X_\theta;$
\item[(2)] the Brauer\textendash Manin obstruction is the only obstruction to weak approximation for rational points (or 0-cycles of degree $1$)
on $X_\theta;$
\item[(3)] we suppose \emph{(2)}, moreover, the induced map $\CH_0(X_v)\to \CH_0(\mathbb{P}^1_v)\simeq\mathbb{Z}$
is assumed injective for almost all places $v.$
\end{itemize}
Then for every integer $\delta,$ we have respectively
\begin{itemize}
\item[(1)] the Brauer\textendash Manin obstruction is the only obstruction to the Hasse principle for 0-cycles of degree $\delta$
on $X;$
\item[(2)] the Brauer\textendash Manin obstruction is the only obstruction to weak approximation for 0-cycles of degree $\delta$
on $X;$
\item[(3)] the sequence $(\E)$ is exact for $X.$
\end{itemize}
\end{thm}

\begin{rem}\label{remark on hypotheses}
The induced map $\CH_0(X_v)\to \CH_0(\mathbb{P}^1_v)\simeq\mathbb{Z}$ is the degree map, it is injective
for almost all $v$ if $X_\eta$ is assumed rationally connected thanks to a theorem of
Koll\'ar\textendash Szab\'o \cite[Thm. 5]{Kollar-Szabo}. It is also well-known that if $X_\eta$ is rationally connected,
the hypothesis (gen) is automatically satisfied, \emph{cf.} \cite[II Cor. 3.4]{Br}, the proof of
\cite[Prop. 2.11]{CTRaskind}, and \cite[Cor. 4.18]{Debarre}.

Smooth proper models of the equation $N_{K/k}(\overrightarrow{\textbf{x}})=P(t)$ with $K/k$ an abelian extension satisfy
the hypothesis (ab-sp), \emph{cf.} \cite[Lem. 3.2]{Liang5}.
These models are fibred over $\mathbb{P}^1$ by the parameter $t,$ almost all fibres are smooth compactifications
of torsors under algebraic tori, the arithmetic hypothesis (1) and (2) are satisfied.

Concerning the hypothesis $(\Br)$, we have a more detailed discussion in \S \ref{application}. One of the applications is the following corollary concerning  varieties discussed in \cite{Wei}.
\end{rem}

\begin{cor}\label{maincor}
Let $K/k$ be an abelian extension with Galois group $\Z/n\Z\times\Z/n\Z.$ Let $P(t)$ be an irreducible polynomial over $k$ and $L=k[t]/(P(t)).$ Suppose that $L$ contains a degree $n$ cyclic subfield of $K$. Then for any smooth proper model $X$ of the equation $N_{K/k}(\overrightarrow{\textup{\textbf{x}}})=P(t),$ the sequence $(\E)$ is exact.
\end{cor}

\section{Preliminaries to the proof}\label{preliminaries}
At first, we state some preliminaries and give references of each statement.
In the subsection \S \ref{Hilbert}, we prove a variant of Hilbert's irreducibility theorem, which
is one of the main ingredients of this work.

\subsection{Formal lemma}

\begin{lem}[{\cite[Lem. 4.5]{CT-Sk-SD}}, original version {\cite[Cor. 2.6.1]{Harari}}]\label{formal lemma}
Let $X$ be a smooth proper geometrically integral variety defined over a number field $k.$ Let
$U$ be a non-empty open subset of $X$ and $\{A_1,\ldots,A_n\}\subset \Br(U)\subset \Br(k(X)).$
We denote by $B$ the intersection of $\Br(X)$ and the subgroup generated by the $A_i$'s in $\Br(k(X)).$

Suppose that for every $v\in\Omega_k,$ there exists a 0-cycle $z_v$ on $X_v$ of degree $\delta$
supported in $U_v$ such that the family $\{z_v\}_{v\in\Omega_k}$ is orthogonal to $B.$

Then, for all finite set $S$ of places of $k,$ there exists a finite set $S'$ of places of $k$ containing
$S,$ and for each $v\in S'$ there exists a 0-cycle $z'_v$ on $U_v$ of degree $\delta$ such that
$$\sum_{v\in S'}\inv_v(\langle A_i,z'_v\rangle_v)=0$$
and moreover $z'_v=z_v$ for all $v\in S.$
\end{lem}

\subsection{Moving lemmas for 0-cycles}

We say that a 0-cycle $z=\sum n_PP$ is \emph{separable} if each non-zero integer $n_P$ equals either $1$ or $-1.$
Let $k$ be a topological field (of characteristic $0$) and $\bar{k}$ be its fixed algebraic closure.
Let $z'$ be an effective 0-cycle of degree $d>0$ on a $k$-variety $V,$ we express it as a sum of closed points $z'=\sum P'_i$
(not necessarily separable, the closed points $P'_i$ may be equal for different $i$).
We say that an effective 0-cycle $z=\sum P_i$ of degree $d$ is \emph{sufficiently close} to $z'$ if
(after a permutation of the indices) we have $k(P_i)=k(P'_i)$ and $P_i$ is sufficiently close to
$P'_i$ in the topological space $V(k(P'_i)).$

\begin{lem}[{\cite[\S 3]{CT05}}]\label{moving 1}
Let $X$ be a integral regular variety defined over a perfect field $k,$ and $U$ be a non-empty open
subset of $X.$ Then every 0-cycle $z$ of $X$ is rationally equivalent on $X$ to a 0-cycle $z'$ supported
in $U.$
\end{lem}

\begin{lem}[{\cite[page 89]{CT-SD}}, {\cite[page 19]{CT-Sk-SD}}]\label{moving 2}
Let $\pi:X\to\mathbb{P}^1$ be a fibration defined over $\mathbb{R},$ $\mathbb{C},$ or a $p$-adic field.
Suppose that $X$ is smooth integral. Let $D$ be a finite set of closed points of $\mathbb{P}^1,$
and $X_0$ be a non-empty Zariski open subset of $X.$

Then for every effective 0-cycle $z\neq0$ supported in $X_0,$ there exists a separable effective 0-cycle $z'$
supported in $X_0$ such that $z'$ is sufficiently close to $z$ and such that
$\pi_*(z')$ is separable and supported outside $D.$ The 0-cycles $\pi_*(z)$ and
$\pi_*(z')$ are rationally equivalent on $\mathbb{P}^1.$
\end{lem}

\subsection{Comparison of Brauer groups}
The following proposition was stated originally for rational points $\theta,$ but the whole proof works
for closed points.
\begin{prop}[{\cite[Thm. 3.5.1]{Harari}}, {\cite[Thm. 2.3.1]{Harari2}}]\label{comparison Br}
Let $X\to\mathbb{P}^1$ be a fibration defined over a number field $k.$ Suppose that
$\Pic(X_{\bar{\eta}})$ is torsion-free and $\Br(X_{\bar{\eta}})$ is finite.

Then there exists a generalized Hilbertian subset $\textsf{Hil}\subset\mathbb{P}^1$ such that
for all $\theta\in \textsf{Hil},$ the specialization
$$sp_\theta:\frac{\Br(X_\eta)}{\Br(k(t))}\to\frac{\Br(X_\theta)}{\Br(k(\theta))}$$
is an isomorphism of finite abelian groups.
\end{prop}

\subsection{Hilbert's irreducibility theorem}\label{Hilbert}

The following proposition is a variant of Hilbert's irreducibility theorem.
This is a crucial step in the proof of the main theorem.
In order to prove the proposition, we follow the strategy of Ekedahl \cite{Ekedahl},
of which more detailed arguments were given by Harari \cite{Harari}.
We combine the method of Colliot-Th\'el\`ene \cite{CT99} to deal with 0-cycles on higher genus curves.

\begin{prop}\label{Hilbert's irreducibility}
Let $k$ be a number field and  $\textsf{Hil}$ be a generalized Hilbertian subset of $\mathbb{P}^1_{/k}.$
Let $P_i$ $(i=1,\ldots,n)$ be closed points of $\mathbb{A}^1=\mathbb{P}^1\setminus\{\infty\}$ of residue field
$k_i.$

Let $S$ be a finite set of places of $k$ and $z_v$ be an effective separable
0-cycle of degree $d>\sum_{i=1}^n[k_i:k]-2$ on $\mathbb{P}^1_v$ supported in $\mathbb{A}^1\setminus\bigsqcup_{i=1}^n P_i$
for each $v\in S.$

Then, given a finite non-trivial extension $F$ of $k,$ there exists
\begin{itemize}
\item[-] an infinite set $I$ of places of $k,$
\item[-] a closed point $\theta\in\mathbb{A}^1$ of degree $d$ (defined by an irreducible polynomial $f\in k[t]$),
\item[-] a place $w_i\in\Omega_{k_i}\setminus S\otimes_kk_i$  for each $i,$
\end{itemize}
satisfying the following conditions
\begin{itemize}
\item[(1)] each place in $I$ splits completely in $F,$
\item[(2)] $\theta\in \textsf{Hil},$
\item[(3)] as a 0-cycle, $\theta$ is sufficiently close to $z_v$ for all $v\in S,$
\item[(4)] as a $k(\theta)$-point of $\mathbb{A}^1,$ $\theta$ is an $S\cup I$-integer,
\item[(5)] for each $i,$ we have $w_i(f(P_i))=1$ and $w(f(P_i))=0$ for all places
            $w\in\Omega_{k_i}\setminus(S\cup I)\otimes_kk_i$ different from $w_i.$
\end{itemize}
\end{prop}

\begin{proof}
The generalized Hilbertian subset $\textsf{Hil}\subset U\subset\mathbb{P}^1$
is defined by a finite morphism $Z\to\mathbb{P}^1$ \'etale over $U,$
where $Z$ an integral $k$-variety. Take $K'$ to be the Galois closure of the extension
$k(Z)/k(\mathbb{P}^1)$ of function fields. Let $Z'$ be the normal curve having $K'$ as its function field,
then the composite of finite morphisms $Z'\to Z\to\mathbb{P}^1$ defines a generalized
Hilbertian $\textsf{Hil}'$ contained in $\textsf{Hil}.$ We are going to find a closed point $\theta\in \textsf{Hil}'.$

Let $k'$ be the algebraic closure of $k$ in $K'.$ By shrinking $U$ if necessary, we may assume that
$U'=Z'\times_{\mathbb{P}^1}U$ is \'etale surjective over $U,$ the cover $U'\to U$ is Galois of group $G=\Gal(K'/k(t)),$
the open $U$ is contained in $\mathbb{A}^1,$ and moreover the $k$-morphism $U'\to U$ factorizes through
$U_{k'}\to U$ by a $k'$-morphism, \emph{cf.} \cite[Lem. 1.3]{Liang3}. Then $U'\to U_{k'}$ is a Galois
cover of group $H=\Gal(K'/k'(t)).$
By enlarging $S$ if necessary, we may assume that all these extend to smooth integral models
$\mathcal{U}'\to\mathcal{U}_{O_{k',S'}}\to\mathcal{U}$ over $O_{k,S}.$
Here and from now on, we write $S'=S\otimes_kk'$ and
$\mathcal{V}=\mathcal{U}_{O_{k',S'}}=\mathcal{U}\times_{O_{k,S}}O_{k',S'}.$

Given a finite non-trivial extension $F/k,$ take $I\subset\Omega_k$ to be the set of places of $k$ that split
completely in the compositum $F\cdot k',$ according to Chebotarev's density theorem $I$ is infinite.

Let $\textbf{E}$ be the finite set of conjugation classes of $H=\Gal(\mathcal{U}'/\mathcal{V}).$
Since $U'$ is geometrically integral over $k',$ the geometric Chebotarev's density theorem
(\emph{cf.} \cite[Lem. 1.2]{Ekedahl}) allow us to construct an injection
$$\gamma:\textbf{E}\To((\Omega_k\setminus S)\cap I)\otimes_kk'$$
such that for every $c\in\textbf{E}$ there exists a point of finite residue field
$\bar{x}_c\in\mathcal{V}(k'(\gamma(c)))$ with associated Frobenius element
$Frob_{\bar{x}_c}$ belonging to the class $c.$
Moreover, we can require that after restricting to $k$ the images of different classes $c\in\textbf{E}$
are different from each other.

For each $v\in S,$ the effective 0-cycle $z_v$ is defined by a separable polynomial
$f_v\in k_v[t],$ in other words $div_{\mathbb{P}^1_v}(f_v)=z_v-d\infty.$

We denote $v_c$ the place of $k$ below $w'=\gamma(c),$ as $v_c$ belongs to $I$ we have $k_{v_c}=k'_{w'}$ and $k(v_c)=k'(w').$
By Hensel's lemma, the point $\bar{x}_c$ lifts to a point
$x_c\in\mathcal{V}(O_{w'})\subset U_{k'}(k'_{w'})=U(k_{v_c})$
where $O_{w'}$ is the ring of integers of the local field $k'_{w'}.$
According to the lemma \ref{closed point} below, there exists a closed point $x'_c$ of
$U$ of degree $d-1$ different from $x_c$ and from the $P_i$'s.
We write $z_{v_c}=x_c+x'_c$ and $div_{\mathbb{P}^1_{v_c}}(f_{v_c})=z_{v_c}-d\infty$ for a polynomial $f_{v_c}\in k_{v_c}[t].$

Similarly, we take a place $v_0\in I\setminus S$ away from the $v_c$'s and we write
$div_{\mathbb{P}^1_{v_0}}(f_{v_0})=z_{v_0}-d\infty$ with a closed point $z_{v_0}\in U$ different from the $P_i$'s and
an irreducible polynomial $f_{v_0}\in k_{v_0}[t]$ of degree $d.$

We consider the following exact sequence
$$0\To\mathcal{O}_{\mathbb{P}^1}(d\infty-\sum_{i=1}^nP_i)\To\mathcal{O}_{\mathbb{P}^1}(d\infty)\To\bigoplus_{i=1}^nP_i\to0$$
where by abuse of notation $P_i$ denotes the skyscraper sheaf on $\mathbb{P}^1$ supported at $P_i.$
As $d>\sum_{i=1}^ndeg(P_i)-2,$ we have $H^1(\mathbb{P}^1,\mathcal{O}_{\mathbb{P}^1}(d\infty-\sum_{i=1}^nP_i))=0$
by Serre's duality and hence an exact sequence of global sections
$$0\To\Gamma(\mathbb{P}^1,\mathcal{O}_{\mathbb{P}^1}(d\infty-\sum_{i=1}^nP_i))\To\Gamma(\mathbb{P}^1,\mathcal{O}_{\mathbb{P}^1}(d\infty))\buildrel{ev}\over\To\bigoplus_{i=1}^nk_i\to0.$$
Here the $k$-linear map $ev$ is the evaluation of a polynomial (of degree $d$) at points $P_i.$
We fix a $k$-linear section $\sigma=\bigoplus_{i=1}^n\sigma_i$ of $ev.$
By enlarging $S$ if necessary, we may also assume that entries of the matrix of the linear map
$\sigma_i$ are all $S$-integers.

For $v\in S\cup\{v_c\}_{c\in\textbf{E}}\cup\{v_0\},$ the polynomial $f_v$ is written in a unique way as
$$f_v=f_{0v}+\sum_i\sigma_i(f_v(P_i))$$
with $f_{0v}\in\Gamma(\mathbb{P}^1,\mathcal{O}_{\mathbb{P}^1}(d\infty-\sum_{i=1}^nP_i))\otimes_kk_v$ a polynomial
of degree $d$ such that $f_{0v}(P_i)=0$ for all $i.$

\emph{A priori} $\rho_{i,v}=f_v(P_i)$ lies in $k_i\otimes_kk_v,$
we have $\rho_{i,v}\in (k_i\otimes_kk_v)^*$ since $P_i\notin supp(z_v).$
Note that $I\setminus (S\cup\{v_c\}_{c\in\textbf{E}}\cup\{v_0\})$ is infinite.
Thanks to a generalized Dirichlet's theorem (\emph{cf.} \cite[Cor. 4.4]{Sansuc82}),
for each $i$ there exists
an element $\rho_i\in k^*_i$ and a place $w_i\in\Omega_{k_i}$ outside $S$ such that
\begin{itemize}
\item[-] $\rho_i$ is sufficiently close to $\rho_{i,v}$ for all $v\in S\cup\{v_c\}_{c\in\textbf{E}}\cup\{v_0\},$
\item[-] $w_i(\rho_i)=1,$
\item[-] $\rho_i$ is a unit outside $\{w_i\}\cup(S\cup I)\otimes_kk_i.$
\end{itemize}

It suffices to find a closed point $\theta\in\mathbb{A}^1$ defined by $f\in k[t]$
having the desired properties (2)(3)(4) and such that $f(P_i)=\rho_i$ for all $i.$

Strong approximation property applied to the finite dimensional $k$-linear space
$\Gamma(\mathbb{P}^1,\mathcal{O}_{\mathbb{P}^1}(d\infty-\sum_{i=1}^nP_i))$
give us a polynomial $f_0$ of degree $d$ with coefficients in $O_{S\cup I}$ such that
$f_0(P_i)=0$ for all $i$ and such that $f_0$ is sufficiently close to $f_{0v}$ for all
$v\in S\cup\{v_c\}_{c\in\textbf{E}}\cup\{v_0\}.$
Then the degree $d$ polynomial $f=f_0+\sum_i\sigma_i(\rho_i)\in k[t]$ is sufficiently close to
$f_v$ with $S\cup I$-integral coefficients.
As $f_{v_0}$ is irreducible, Krasner's lemma implies that $f$ is irreducible over
$k_{v_0}$ and \emph{a fortiori} irreducible over $k.$ Moreover $v_0$ splits completely in $k',$
the field $k'$ is contained in $k_{v_0},$ thus $f$ is irreducible over $k'.$
We write $div_{\mathbb{P}^1}(f)=\theta-d\infty,$
then the 0-cycle $\theta$ is actually a closed point of $U,$ and
its preimage $\theta'$ in $U_{k'}$ is also a closed point.
Let $L=k(\theta)$ be the residue field of $\theta,$ then $L'=L\otimes_kk'$ is a field and is the residue field
of $\theta'.$
As a 0-cycle $\theta$ is sufficiently close to
$z_{v_c}=x_c+x'_c$ for all $c\in\textbf{E}.$ This means that there exists a place $w$ of
$L=k(\theta)$ above $v_c$ such that $L_w/k_{v_c}$ is a trivial extension and moreover
the image $\theta_w$ of $\theta$ under $U(L)\to U(L_w)$ is sufficiently close to $x_c.$
Hence $\theta$ is also integral (for the model $\mathcal{U}$) at $w,$ and it reduction
modulo $w$ is $\bar{x}_c\in\mathcal{U}(L(w))=\mathcal{U}(k(v_c)).$

Recall that $w'=\gamma(c)\in \Omega_{k'}$ is a place over $v_c\in\Omega_k$ such that $k'_{w'}/k_{v_c}$
is a trivial extension. Let $\lambda$ be a place of $L'$ above $w\in\Omega_L$ and above $w'\in\Omega_{k'}.$
The point $\theta'\in U_{k'}(L'_\lambda)$ is actually an integral point (for the model $\mathcal{V}$)
of reduction modulo $\lambda$ exactly $\bar{x}_c\in\mathcal{V}(L'(\lambda))=\mathcal{V}(k'(w')).$
Consider the map $\Gal(\bar{L}'/L')\to H=\Gal(U'/U_{k'})$ defined by a chosen lifting of
$Spec(\bar{L}')\to Spec(L')\buildrel{\theta'}\over\To U_{k'}$ to $U',$
the fact that the reduction of $\theta'$ is $\bar{x}_c$ means that the conjugate class $c\in\textbf{E}$ of the
Frobenius element $Frob_{\bar{x}_c}\in H$ intersects with the image of $\Gal(\bar{L}'/L').$
This holds for all $c\in\textbf{E},$ then $\Gal(\bar{L}'/L')$ maps surjectively to $H$
by a well known result of finite group theory, \emph{cf.} \cite[Lem. 1.1]{Ekedahl}.
The preimage of $\theta'$ under $U'\to U_{k'}$ is then connected by the theory of \'etale fundamental groups.
The point $\theta\in U$ has connected preimage in $U'$ and so $\theta\in \textsf{Hil}'\subset\textsf{Hil}.$
And it is an $S\cup I$-integer since $f$ has coefficients in $O_{S\cup I}.$
\end{proof}

\begin{lem}\label{closed point}
Let $k$ be a $p$-adic local field and $D$ a finite set of closed points of $\mathbb{P}^1_{/k}.$
Then for every positive integer $n,$ there exists a closed point of $\mathbb{P}^1\setminus D$
of degree $n.$
\end{lem}

\begin{proof}
Over a local field, there exists infinite many irreducible polynomials of a fixed degree $n.$
\end{proof}

\subsection{Existence of local points}

\begin{lem}[{\cite[Lem. 1.2]{CT-Sk-SD}}]\label{LangWeil}
Let $k$ be a number field and $Spec(O)$ a non-empty open set of the ring of integers of $k.$
Let $\Pi:\mathcal{X}\to\mathbb{P}^1_O$ be a flat, projective morphism with $\mathcal{X}$ regular
and smooth over $O.$ Let $\pi:X\to\mathbb{P}^1$ be the restriction of $\Pi$ over $Spec(k).$
Let $T\subset\mathbb{P}^1_O$ be a closed subset, finite and \'etale over $O,$ such that
fibres of $\Pi$ above points not in $T$ are split. Let $T=\bigcup_{i=1}^nT_i$ be the decomposition
of $T$ into irreducible closed subsets, and lest $k_i$ be the field of fractions of $T_i.$

After inverting finitely many primes in $O,$ the following holds.

(a) Given any closed point $u\in\mathbb{P}^1_O,$ if the fibre $\mathcal{X}_u$ over the finite field
$k(u)$ is split, then it contains a smooth $k(u)$-point.

(b) Given any closed point $\theta\in\mathbb{P}^1,$ with Zariski closure
$Spec(\tilde{O})\simeq\tilde{\theta}\subset\mathbb{P}^1_O,$ where $\tilde{O}/O$ is finite
with $Frac(\tilde{O})=k(\theta),$ we denote by $\tilde{O}'$ the integral closure of $\tilde{O}$ in
$k(\theta).$ If $u\in\tilde{\theta}$ is a closed point such that $\mathcal{X}_u/k(u)$
is split, then $X_\theta$ contains a smooth $k(\theta)_v$-point where $v$ is a place of $k(\theta)$
(associated to a closed point of $Spec(\tilde{O}')$) above $u.$

(c) Let $u$ belong to one of the $T_i$'s, thus defining a place $v_i$ of $k_i.$ Assume that there exists
a irreducible component $Z$ of the fibre of $\Pi$ at $P_i=T_i\times_Ok$ which has multiplicity one.
Let $k'_i$ denote the algebraic closure of $k_i$ in the function field of $Z.$ If
the place $v_i$ splits completely in the ring of integers of $k'_i,$ then $\mathcal{X}_u/k(u)$ is
split.

(d) Assume that for each $i$ there exists at least one irreducible component of $\Pi^{-1}(P_i)$
which has multiplicity one. Then given any finite extension $F_0/k$ there exists a finite extension
$F$ containing $F_0$ such that for each place $v\in\Omega_k$ splits completely in $F$ the induced
map $X(L)\to\mathbb{P}^1(L)$ is surjective for all finite extensions $L/k_v.$
\end{lem}

\begin{proof}
This is a version for closed points of \cite[Lem. 1.2]{CT-Sk-SD}, as indicated in \cite[page 20]{CT-Sk-SD}
the same proof works for this case. The last conclusion is slightly different from the original one,
but the original proof applies.

We remark that by taking a certain non-empty open subset $Y$ of $X$ and extending it to an integral model over $O,$
the same argument as in the proof (Lang-Weil's estimation and Hensel's lemma) implies that
we can moreover require that the rational points over local fields that obtained are situated inside $Y$ and are
the liftings (by Hensel's lemma) of rational points of finite residue fields of the integral model.
We will use this more precise statement in the proof of the main theorem.
\end{proof}

\section{Proof of the main theorem}\label{proof}

The whole section is devoted to proving the main theorem \ref{main_thm}.

\begin{proof}[Proof of Theorem \ref{main_thm}]
We are going to give a complete proof of the conclusion (2) with assumption (2) concerning 0-cycles of degree $1,$ similar
argument works with assumption (2) concerning rational points with the help of \cite[Lem. 1.8]{Wittenberg}.
The same method proves (1).
Note that if $\CH_0(X_v)\to \CH_0(\mathbb{P}^1_v)$ is assumed injective for almost all places, the first conclusion
in (3) is reduced to the conclusion (2), and the exactness of $(\E)$ is deduced by \cite[Prop. 3.1]{Wittenberg}
applied with the base curve the projective line.

\bigskip
\noindent\emph{Outline of the proof.}

After some preparations of \emph{Steps} \ref{construction E}-\ref{integral models}, in \emph{Step}
\ref{application-formal-moving lemma} we start from $\{z_v\}\bot \Br(X).$ By Harari's formal lemma, we
reduce the pairing to a finite sum over $S_2$ containing the places in $S$ over which we need to approximate the local 0-cycles.
We argue with moving lemmas to reduce to separable effective 0-cycles without changing the Brauer\textendash Manin pairing.
In \emph{Step} \ref{application hilbert irred} we apply Hilbert's irreducibility theorem (Prop. \ref{Hilbert's irreducibility})
to find a closed point $\theta\in\mathbb{P}^1$ close enough to the projections of the separable effective local 0-cycles.
In \emph{Step} \ref{local points}, we verify that $X_\theta$ has points $M_w$ locally everywhere.
For the places in $S_2,$ we apply the implicit function theorem.
For the places outside $S_2,$ We make use of the orthogonality given by
the vertical elements $A_{i,j}$ in the Brauer group, these local points are essentially given
by Hensel's lemma and Lang-Weil's estimation (Lem. \ref{LangWeil}).
In \emph{Step} \ref{orthogonal to Br}, we compute each local term of the Brauer\textendash Manin pairing and we
modify certain $M_w$ (by applying the geometric Chebotarev density theorem)
in order to obtain, on the fibre $X_\theta,$ the orthogonality to the Brauer group.
The modified $M_w$'s are above carefully chosen places $v_r$ in \emph{Step} \ref{construction of assisting places}.
The finite set $E$ constructed in \emph{Step} \ref{construction E} is used to verify one of the hypotheses of Chebotarev's
theorem.

\bigskip
\noindent\emph{Aim.}

In order to prove (2), we fix an integer $\delta,$ a positive integer $N,$ and a finite set $S$ of places of $k.$
Given a family of local 0-cycle $\{z_v\}_{\Omega_k}$ of degree $\delta$ orthogonal to $\Br(X),$ we need to
find a global 0-cycle $z$ such that $z$ and $z_v$ have the same image in $\CH_0(X_v)/N$ for all $v\in S.$

\bigskip

\noindent\emph{Notation.}

We preserve the notation $Y,$ $Z,$ $X_r,$ and $\Lambda,$ at the beginning of \S \ref{main_result},
the complement of $Z=\bigsqcup_{r=1}^lX_r$ is an open $Y\subset X$ with $\Lambda\subset \Br(Y)$ generating
$\Br(X_\eta)/\Br(k(t)).$
Let $P_i$ $(i=1,\ldots,n)$ be the closed points of $\mathbb{P}^1$ such that the fibre
$X_i=X_{P_i}$ is not split. Then for all $i,$ the point $P_i$ is contained in the open subset $V=\pi(Y)$
by the hypothesis $(\Br)$.
Let $U$ be a non-empty open subset of $\mathbb{P}^1$ such that all fibres of $\pi$ over $U$ is smooth and geometrically
integral.
Chose a $k$-point $\infty$ in $U\cap V$ different from those $m_r$'s. The open $V_0=V\setminus\{\infty\}$
is contained in $\mathbb{A}^1=\mathbb{P}^1\setminus\{\infty\}.$
The fibre $X_\infty$ is a smooth $k$-variety of function field $K_\infty=k(X_\infty).$
We write $Y_0=Y\setminus X_\infty=\pi^{-1}(V_0).$

For each $1\leqslant r\leqslant l,$ the fibre $X_r$ at the closed point $m_r\in\mathbb{A}^1$ is split.
We fix an irreducible component $X_r^{irr}$ of $X_r$ of multiplicity $1$ such that it is geometrically
integral over $k_r=k(m_r).$ The point $m_r$ is defined by a monic irreducible polynomial $Q_r(t)\in k[t]$
and its residue field $k_r=k[t]/(Q_r(t)).$ Denote by $K_r=k_r(X_r^{irr})$ the function field of
$X_r^{irr}.$ Let $A$ be an element of $\Br(k(X)),$ we denote by $\partial_{A,r}\in H^1(K_r,\mathbb{Q}/\mathbb{Z})$
its residue at the generic point of $X_r^{irr}.$ The subgroup of $H^1(K_r,\mathbb{Q}/\mathbb{Z})$ generated
by the elements $\partial_{A,r}$ $(A\in\Lambda)$ is of the form $G_r=H^1(\Gal(K'_r/K_r),\mathbb{Q}/\mathbb{Z})$
where $K'_r$ is a finite abelian extension of $K_r.$ Let $k'_r$ be the algebraic closure of $k_r$
in $K'_r.$ Since $k_r$ is algebraically closed in $K_r,$ the subgroup
$G'_r=H^1(\Gal(K_rk'_r/K_r),\mathbb{Q}/\mathbb{Z})$ of $G_r$ is isomorphic
to $H^1(\Gal(k'_r/k_r),\mathbb{Q}/\mathbb{Z}).$ Let $Y_r$ be a smooth open subset of $X_r^{irr}$ with empty
intersection with other irreducible components of $X_r.$ By shrinking $Y_r$ if necessary, we may
assume that there exists a finite \'etale connected Galois cover $W_r$ of $Y_r$ of Galois group
$\Gal(K'_r/K_r)$ such that $W_r\to Y_r$ factorizes through $Y'_r=Y_r\times_{k_r}k'_r\to Y_r.$
Then $W_r$ is a variety geometrically integral over $k'_r.$
Moreover, we may also assume that the elements of
$H^1(\Gal(K'_r/K_r),\mathbb{Q}/\mathbb{Z})\subset H^1(K_r,\mathbb{Q}/\mathbb{Z})$ come from
elements of $H^1_{\scriptsize{\mbox{\'et}}}(Y_r,\mathbb{Q}/\mathbb{Z}).$

We also denote by $\partial_{A,\infty}\in H^1(K_\infty,\mathbb{Q}/\mathbb{Z})$ the residue of $A\in \Br(k(X))$ at the generic point of $X_{\infty}.$
Similarly we have the notation $K'_\infty,$ $G_\infty,$ $k'_\infty,$ $Y_{\infty},$ $W_{\infty}$ and $Y'_{\infty}.$

For each $1\leqslant i\leqslant n,$ let $k_i$ be the residue field of $P_i.$ The point $P_i$ gives
a rational point $e_i\in\mathbb{A}^1(k_i)=k_i.$ We set $g'_i=t-e_i\in k_i[t]$ and $g_i=N_{k_i(t)/k(t)}(g'_i)\in k[t],$
then $k_i=k[t]/(g_i).$ We fix an irreducible component $X_i^{irr}$ of the fibre
$X_i$ at $P_i$ of multiplicity $1$ which splits after an abelian extension of $k_i.$
Denote by $K_i=k_i(X_i^{irr})$ the function field of $X_i^{irr}$ and by $k'_i$ the algebraic
closure of $k_i$ in $K_i.$ The field $k'_i$ is a finite abelian extension of $k_i,$
it is a compositum of finitely many cyclic extension $k_{i,j}$ of $k_i.$
By fixing a character $\chi$ of the cyclic group $\Gal(k_{i,j}(t)/k_i(t)),$ we define
$(k_{i,j}/k_i,g'_i)=(\chi,g'_i)$ an element of $\Br(k_i(t))$ via cup product, and we set
$A_{i,j}=cores_{k_i(t)/k(t)}(k_{i,j}/k_i,g'_i)\in \Br(k(t)),$ \emph{cf.} \cite[\S 1]{CT-SD}.
There exists a closed subset $D\subset \mathbb{P}^1$ containing all $P_i$'s such that
$A_{i,j}\in \Br(\mathbb{P}^1\setminus D)$ for all $i,j.$

Replacing by a smaller generalized Hilbertian subset if necessary, according to Proposition \ref{comparison Br},
we may assume that for every closed point $\theta\in \textsf{Hil}\subset\mathbb{A}^1$ the specialization
to the smooth fibre $X_\theta$
$$sp_\theta:\frac{\Br(X_\eta)}{\Br(k(t))}\To\frac{\Br(X_\theta)}{\Br(k(\theta))}$$
is an isomorphism, hence $\Lambda$ maps onto the Brauer group of the fibre $X_\theta$
(up to constant elements which do not
contribute to the Brauer\textendash Manin obstruction on the fibre $X_\theta$).

\medskip
\begin{step}\label{construction E} \emph{Construction of a finite set $E$ of elements in $\Br(k(t)).$}

The finite set $E$ is constructed in order to verify (in \emph{Step} \ref{computation for w r})
one of the hypotheses of the geometric Chebotarev's density theorem applied in \emph{Step} \ref{modification}.

Consider Faddeev's exact sequence (\emph{cf.} \cite[Cor. 6.4.6]{GilleSzamuely})
$$0\To \Br(k)\To \Br(k(t))\buildrel{\partial_\theta}\over\To\bigoplus_\theta H^1(k(\theta),\mathbb{Q}/\mathbb{Z})\To0$$
where the direct sum is taken over all closed points $\theta$ of $\mathbb{A}^1$ and
$\partial_\theta$ is the residue map at $\theta.$
For each $1\leqslant r\leqslant l,$ we consider the subgroup
$$E'_r=\{\beta\in \Br(k(t));\partial_\theta(\beta)=0\mbox{ if }\theta\neq m_r\mbox{ and }\partial_{m_r}(\beta)\in G'_r=H^1(\Gal(k'_r/k_r),\mathbb{Q}/\mathbb{Z})\}.$$
As $E'_r/\Br(k)$ is a finite group, we take a finite subset $E_r\subset \Br(k(t))$ of its representatives,
then we define $E$ to be the (disjoint) union of $E_r$'s.
Thus $E\subset \Br(V_0)$ and its image in $\Br(k(X))$ is contained in
$\Br(Y_0).$ We often identify $E$ with its image in $\Br(k(X)).$
\end{step}

\medskip
\begin{step}\label{construction F0} \emph{Construction of a finite extension $F_0/k.$}

By Tsen's theorem, the Brauer group of $\bar{k}(t)$ is trivial. Therefore there exists a finite extension
$k'$ of $k$ such that the restriction of $E$ in $\Br(k'(t))$ is $0.$
We fix a finite extension $F_0$ of $k$ containing $k',$ $k'_\infty,$ $k'_i$ for all $i,$ and $k_r$ for all $r.$
This is for the future computation of the Brauer\textendash Manin pairing in \emph{Step} \ref{computation for w infty}.
\end{step}

\medskip
\begin{step}\label{integral models} \emph{Extension to integral models.}

There exists a finite set $S_1$ of places of $k$ containing $S$
such that all the following statements hold.
\begin{itemize}
\item[-] $S_1$ contains all archimedean places of $k$.
\item[-] The variety $X$ extends to an integral model $\mathcal{X}$ over $O_{k,S_1},$
        the morphism $\pi$ extends to a flat morphism $\Pi:\mathcal{X}\to\mathbb{P}^1_{O_{k,S_1}}.$
\item[-] For each $1\leqslant i\leqslant n,$ the polynomial $g_i$ has coefficients in $O_{k,S_1}.$
        The Zariski closure $T_i$ of $P_i$ in $\mathbb{P}^1_{O_{k,S_1}}$
         (it is also the closure of the closed subvariety of $\mathbb{A}^1_{O_{k,S_1}}$ defined by $g_i=0$)
         is \'etale over $Spec(O_{k,S_1}).$
         Non-split fibres of $\Pi$ are situated above $T=\bigcup T_i.$ Moreover $T_i$ and $T_{i'}$ are disjoint if $i\neq i'.$
\item[-] For each $1\leqslant r\leqslant l,$ we denote by $\widetilde{m}_r$ the Zariski closure of $m_r$
        in $\mathbb{P}^1_{O_{k,S_1}}$ and
        by $\mathcal{X}_r$ the fibre $\Pi^{-1}(\widetilde{m}_r).$
        The polynomial $Q_r(t)$ defining $m_r$ has coefficients in $O_{k,S_1}.$
        The schemes $\widetilde{m}_r$ and $\widetilde{m}_{r'}$
        are disjoint if $r\neq r'.$ The open subset $Y_r\subset X_r$ extends
        to an open $\mathcal{Y}_r\subset\mathcal{X}_r$ such that the elements of $H^1(\Gal(K'_r/K_r),\mathbb{Q}/\mathbb{Z})$
        come from elements of $H^1_{\scriptsize{\mbox{\'et}}}(\mathcal{Y}_r,\mathbb{Q}/\mathbb{Z}).$
        The scheme $\mathcal{Y}_r$ is smooth over $O_{k,S_1}.$
\item[-] We denote by $\widetilde{\infty}$ the Zariski closure of $\infty$ in $\mathbb{P}^1_{O_{k,S_1}}$ and
        by $\mathcal{X}_\infty$ the fibre $\Pi^{-1}(\widetilde{\infty}).$ The scheme $\widetilde{\infty}$ is disjoint
        with $\widetilde{m}_r$ for all $r.$ The open subset $Y_\infty\subset X_\infty$
        extends to an open $\mathcal{Y}_\infty\subset\mathcal{X}_\infty$
        such that the elements $\partial_{A,\infty}$ ($A\in\Lambda\cup E$)
        come from elements of $H^1_{\scriptsize{\mbox{\'et}}}(\mathcal{Y}_\infty,\mathbb{Q}/\mathbb{Z}).$
        The scheme $\mathcal{Y}_\infty$ is smooth over $O_{k,S_1}.$
\item[-] We set $\mathcal{V}=\mathbb{P}^1_{O_{k,S_1}}\setminus\bigcup\widetilde{m}_r,$
        $\mathcal{V}_0=\mathcal{V}\setminus\widetilde{\infty},$
        $\mathcal{Y}=\mathcal{X}\setminus\bigcup\mathcal{X}_r,$ and $\mathcal{Y}_0=\mathcal{Y}\setminus\mathcal{X}_\infty.$
         The elements of $\Lambda\cup E\subset \Br(Y_0)$ come from elements of $\Br(\mathcal{Y}_0).$
\item[-] The finite connected \'etale Galois cover $W_r\to Y_r$ extends to an $O_{k_r,S_1}$-morphism $\mathcal{W}_r\to\mathcal{Y}_r$
        factorizing through $\mathcal{Y}'_r=\mathcal{Y}_r\times_{O_{k_r,S_1}}O_{k'_r,S_1}\to\mathcal{Y}_r$
        such that $\mathcal{W}_r\to\mathcal{Y}_r$ is a Galois cover of group $\Gal(K'_r/K_r)$ and
        $\mathcal{W}_r\to\mathcal{Y}'_r$ is a Galois cover of group $\Gal(K'_r/K_rk'_r).$ And similarly for $r=\infty.$
\item[-] Lemma \ref{LangWeil} applies to $O=O_{k,S_1}.$
\item[-] For $1\leqslant r\leqslant l,$ the geometric Chebotarev's density theorem \cite[Lem. 1.2]{Ekedahl} applies to
            the $O_{k'_r,S_1}$-morphism $\mathcal{W}_r\to\mathcal{Y}'_r$, \emph{i.e.}
            the cardinality of $k'_r(v_r)$ for all $v_r\in\Omega_{k'_r}\setminus S_1\otimes_kk'_r$
            is large enough to guarantee the existence of a $k'_r(v_r)$-point of $\mathcal{Y}'_r$ having its Frobenius element
            in a given conjugacy class of $\Gal(\mathcal{W}_r/\mathcal{Y}'_r)=\Gal(K'_r/K_rk'_r).$
\end{itemize}
\end{step}

\medskip
\begin{step}\label{application-formal-moving lemma} \emph{Application of formal lemma and moving lemmas.}

Fix a closed point $z_0$ of $X$ such that the 0-cycle $y_0=\pi_*(z_0)$ is a closed point having
the same residue field as $z_0$ and such that $y_0$ is different from $\infty$ and the points in $D,$
denote by $d_0$ the degree of $z_0.$

Let $\{z_v\}_{v\in\Omega_k}$ be a family of local 0-cycles of degree $\delta$ orthogonal to $\Br(X).$
By Lemma \ref{moving 1}, we may assume that each $z_v$ is supported in $Y_0$ and $\pi_*(z_v)$ is supported
disjoint from $D.$
According to the formal lemma \ref{formal lemma} applied to $\pi^*(A_{i,j})\in \Br(\pi^{-1}(\mathbb{P}^1\setminus D))$
and the elements in $\Lambda\cup E\subseteq \Br(Y_0),$ there exists a
finite set of places $S_2$ of $k$ containing $S_1$ such that and for each $v\in S_2$ a 0-cycle $z'_v$
of degree $\delta$ supported in $Y_0$ and away from fibres above $D$ such that
\begin{itemize}
\item[-] $\sum_{v\in S_2}\inv_v(\langle A_{i,j},z'_v\rangle_v)=0$ for all $A_{i,j};$
\item[-] $\sum_{v\in S_2}\inv_v(\langle A,z'_v\rangle_v)=0$ for all $A\in\Lambda\cup E;$
\item[-] $z'_v=z_v$ for all $v\in S_1.$
\end{itemize}
In the rest of the proof, it suffice to approximate $z'_v$ for every $v\in S_2.$
We write in a unique way $z'_v=z^+_v-z^-_v,$ where $z^+_v$ and $z^-_v$ are effective 0-cycles with
support disjoint from each other.
Recall that we are considering the images of 0-cycles in $\CH_0(X_v)/N,$ let $a$ be a positive integer which
is divisible by $N$ and who annihilates all the elements $A\in\Lambda\cup E$ and $A_{i,j}.$
Then the 0-cycle $z^1_v=z'_v+ad_0z^-_v=z^+_v+(ad_0-1)z^-_v$ is effective of degree
congruent to $\delta$ modulo $ad_0.$ We add to each $z^1_v$ a suitable positive multiple of
$az_0$ and obtain $z^2_v$ of the same degree $d$ for all $v\in S_2.$
Moreover $d\equiv\delta(\mbox{mod }ad_0)$ can be taken to be sufficiently large
such that Proposition \ref{Hilbert's irreducibility} applies.
By Lemma \ref{moving 2}, for each $v\in S_2,$ there exists an effective 0-cycle $\tau_v$ of degree $d$
close enough to $z^2_v$ and such that it is supported in $Y_0$ and away from the
fibres above $D$ and such that $\pi_*(\tau_v)$ is a separable 0-cycle.
By continuity of the Brauer\textendash Manin pairing, note that $a$ annihilates the elements appeared in the pairing, we have
\begin{itemize}
\item[-] $\sum_{v\in S_2}\inv_v(\langle A_{i,j},\tau_v\rangle_v)=0$ for all $A_{i,j};$
\item[-] $\sum_{v\in S_2}\inv_v(\langle A,\tau_v\rangle_v)=0$ for all $A\in\Lambda\cup E.$
\end{itemize}
By \cite[Lem. 1.8]{Wittenberg}, note that $N$ divides $a,$ the 0-cycles $z'_v$ and $\tau_v$ have the same image in
$\CH_0(X_v)/N$ for all $v\in S_2.$
\end{step}

\medskip
\begin{step}\label{construction of assisting places} \emph{Construction of assisting places $v_r.$}

In this \emph{Step} we construct for each $r\in\{1,\ldots,l,\infty\}$ an assisting place $v_r$
with a local 0-cycle, these will give us ability to modify the local points in \emph{Step} \ref{orthogonal to Br}.

Lemma \ref{LangWeil}(d) applied to our fibration with $F_0$ constructed in \emph{Step} \ref{construction F0}
gives us a finite extension $F/k$ containing $F_0.$ Suppose that the generalized Hilbertian subset
$\textsf{Hil}$ is given by a finite morphism $Z\to\mathbb{P}^1,$ let $k'$ be the algebraic
closure of $k$ inside the Galois closure of the field extension $k(Z)/k(t).$

For each $r\in\{1,\ldots,l,\infty\},$ let $v_r$ be a place of $k$ outside $S_2$ and which splits completely in $F\cdot k'.$
We may also suppose that they are different from each others.

As consequence, for $1\leqslant r\leqslant l$ the polynomial $Q_r(t)$ modulo $v_r$ has a simple root in $k(v_r)$
which lifts to a $k_{v_r}$-point $x_r$ of $\mathbb{A}^1$
satisfying $v_r(Q_r(x_r))=1,$ moreover $x_v$ can be chosen different from all the $P_i$'s.
Similarly, we choose $x_\infty\in k^*_{v_\infty}\subset\mathbb{A}^1(k_{v_\infty})$ different
from the $P_i$'s and such that $v_\infty(1/x_\infty)=1.$
Lemma \ref{closed point} permit us to chose a closed point $x'_r$ of $\mathbb{A}^1$ of degree $d-1$ different from the $P_i$'s,
then $x_r+x'_r$ is a separable effective 0-cycle of degree $d.$
\end{step}

\medskip
\begin{step}\label{application hilbert irred} \emph{Application of Hilbert's irreducibility theorem.}

With the field $F,$ the closed points $P_i,$ and the generalized Hilbertian subset $\textsf{Hil}\cap V_0,$
we apply Proposition \ref{Hilbert's irreducibility} to local 0-cycles $\pi_*(\tau_v)$ for $v\in S_2$
as well as $x_r+x'_r$ for $r\in\{1,\ldots,l,\infty\}.$
By Chebotarev's density theorem, we obtain an infinite set $I$ of the places of $k$ which split completely
in $F,$ \emph{a fortiori} split completely in $k_r,$ $k_{i,j}$ and $k'_\infty.$ According to the construction of
$I$ in the proof of \ref{Hilbert's irreducibility}, the places $v_r$ belong to $I$ for $r\in\{1,\ldots,l,\infty\}.$
We also obtain a closed point $\theta\in \textsf{Hil}\subset V_0\subset\mathbb{A}^1$ of degree $d$
sufficiently close to $\pi_*(\tau_v)$ for $v\in S_2$ and sufficiently close to $x_r+x'_r$ for $r\in\{1,\ldots,l,\infty\},$
as a $k(\theta)$-point of $\mathbb{A}^1$
it is an $S_2\cup I$-integer. Moreover, let $f\in k[t]$ be the polynomial defining $\theta$ constructed in
Proposition \ref{Hilbert's irreducibility}, for each $1\leqslant i\leqslant n$ we obtain
a place $w_i$ of $k_i$ away from $S_2\cup I$ such that $w_i(f(P_i))=1$ and $w(f(P_i))=0$ for all
$w\in\Omega_{k_i}\setminus(S_2\cup I)\otimes_kk_i$ different from $w_i.$
\end{step}

\medskip
\begin{step}\label{local points} \emph{The fibre $X_\theta$ has points locally everywhere.}

Let $w\in\Omega_{k(\theta)}$ be a place above $v\in\Omega_k.$

If $v\in S_2,$ the implicit function theorem shows that $X_\theta$ possesses $k(\theta)_w$-points.

If $v\in I,$ Lemma \ref{LangWeil}(d) implies that $X_\theta$ possesses $k(\theta)_w$-points.

If $v\notin S_2\cup I,$ we denote by $\tilde{\theta}\simeq Spec(A)$ the Zariski closure of
$\theta$ in $\mathbb{P}^1_{O_{k,S_2}},$ where $A$ is a finite $O_{k,S_2}$-algebra with fraction field $k(\theta)$
and its integral closure in $k(\theta)$ is $O_{k(\theta),S_2}.$
We fix a place $w$ of $k(\theta)$ above $v,$ it defines closed point $w$ of the normalization
$Spec(O_{k(\theta),S_2})$ of $\tilde{\theta}$ lying above a certain closed point $w_\theta\in\tilde{\theta}.$
Recall that $\tilde{\theta}$ and $T_i$ are locally defined respectively by $f$ and $g_i$
(polynomials with $O_{S_2\cup I}$-integral coefficients).
There are two possible cases.
\begin{itemize}
\item[(i)] If $w_\theta$ is contained in one (unique) of the $T_i$'s. We know that
for $w'\in\Omega_{k_i}\setminus(S_2\cup I)\otimes_kk_i,$ we have
$w'(f(P_i))=0$ except only one possible case where $w'=w_i$ and $w_i\in\Omega_{k_i}\setminus(S_2\cup I)\otimes_kk_i,$
for which we have $w_i(f(P_i))=1.$
The point $w_\theta$ is contained in $T_i$ if and only if the exceptional case happens.
In such a case, considering the intersection $T_i\cap\tilde{{\theta}}$ at the point $w_i,$
the intersection multiplicity equals to $1$
since $w_i(f(P_i))=1.$ Then $w_i,$ viewed as a closed point $w_{\theta}$ of $\tilde{{\theta}},$
must be a regular point of $\tilde{{\theta}}.$ Therefore $w=w_{\theta}=w_i,$
$k_{iw_i}=k({\theta})_w$ and $w(g_i({\theta}))=w_i(f(P_i))=1.$

\item[(ii)] If $w_{\theta}\notin T_i$ for all $i,$ then the fibre $\mathcal{X}_{w_{\theta}}/k(w_{\theta})$
is split by the construction of $T_i,$ thus
$X_{\theta}(k({\theta})_w)\neq\emptyset$ according to Lemma \ref{LangWeil}(b).
In this case, we know that $g_i({\theta})$ is a unit (modulo $w_{\theta}$) in
$k(w_{\theta})\subset k(w)$ since $w_{\theta}\notin T_i\cap\tilde{{\theta}},$ then $w(g_i({\theta}))=0.$
\end{itemize}

Remark that if $w_i\in I\otimes_kk_i,$ case (i) will never happen.
To complete this \emph{Step}, it remains to verify that for $w=w_i\in T_i$ (case (i) if it happens) the fibre
$X_\theta$ possesses $k(\theta)_{w_i}$-points. This will occupy the rest of this \emph{Step} and we may
assume that $w_i\in\Omega_{k_i}\setminus(S_2\cup I)\otimes_kk_i.$

We define $E_i=k_i\otimes_kk({\theta})$ and $F_{i,j}=k_{i,j}\otimes_kk({\theta}).$
Then
$$\langle A_{i,j},{\theta} \rangle _{\mathbb{P}^1}=cores_{k({\theta})/k}cores_{E_i/k({\theta})}(F_{i,j}/E_i,g'_i({\theta}))\in \Br(k)$$
by definition.

By continuity of the Brauer\textendash Manin pairing,
$$\sum_{v\in S_2} \inv_v(\langle A_{i,j},{\theta} \rangle _v)=\sum_{v\in S_2} \inv_v(\langle A_{i,j},\pi_*(\tau_v)\rangle _v)=0,$$
hence
$$\sum_{v\in \Omega_k\setminus S_2} \inv_v(\langle A_{i,j},{\theta} \rangle _v)=0$$ since ${\theta}$ is global.
In other words
$$\sum_{v\in \Omega_k\setminus S_2}\inv_v(cores_{k({\theta})/k}cores_{E_i/k({\theta})}(F_{i,j}/E_i,g'_i({\theta})))=0,$$
$$\sum_{w\in \Omega_{k({\theta})}\setminus S_2\otimes_kk({\theta})}\inv_w(cores_{E_i/k({\theta})}(F_{i,j}/E_i,g'_i({\theta})))=0.$$

We consider a place $w$ of $k(\theta)$ away from $S_2,$ and we are going to calculate each term in the sum above.

If $w\in I\otimes_kk(\theta),$ let $v$ be the place of $k$ below $w.$
By construction, the extension of local fields associated to $k_{i,j}/k_i$ is trivial above the place $v,$
then the extension $F_{i,j}/E_i$ is trivial above the place $w,$
we find that $$\inv_w(cores_{E_i/k({\theta})}(F_{i,j}/E_i,g'_i({\theta})))=0.$$

If $w\notin I\otimes_kk(\theta)$ and $w\neq w_i$ (\emph{i.e.}, the point $w_{\theta}\in\tilde{\theta}$ associated to $w$
is not in $T_i$),
we recall that in this case $w(g_i(\theta))=0,$ then $g_i({\theta})=N_{E_i/k(\theta)}(g'_i(\theta))$ is a unit at $w,$ we also obtain
$$\inv_w(cores_{E_i/k({\theta})}(F_{i,j}/E_i,g'_i({\theta})))=0.$$

Therefore we get finally
$$\mbox{ }\inv_{w_i}(cores_{E_i/k({\theta})}(F_{i,j}/E_i,g'_i({\theta})))=0.\leqno(\star)$$

Consider the natural map $E_i\to E_i\otimes_{k({\theta})}k({\theta})_{w_i},$ where $E_i\otimes_{k({\theta})}k({\theta})_{w_i}$
is a product of extensions of $k({\theta})_{w_i}.$
Note that $w(N_{E_i/k({\theta})}(g'_i({\theta})))=w(g_i({\theta}))$ equals to either  $0$ or $1$ according to $w\neq w_i$
or $w=w_i,$ there is only one of these extensions, denoted by $E_{i,w_i},$
in which the image of $g'_i({\theta})$ is not
a unit but a uniformizer, and moreover, $E_{i,w_i}/k({\theta})_{w_i}$ is trivial.
The equality $(\star)$ implies that $(F_{i,j}/E_i,g'_i({\theta}))\otimes_{E_i}E_{i,w_i}=0,$
we have then for all $j$  the cyclic extension $k_{i,j}/k_i$ is trivial after $\otimes_{E_i}E_{i,w_i}$
since $g'_i({\theta})$ is a uniformizer of $E_{i,w_i},$ we find that
$k'_i/k_i$ is trivial after $\otimes_{E_i}E_{i,w_i}.$
By Lemma \ref{LangWeil}(c), the reduction  $\mathcal{X}_{w_i}/k(w_i)$
of $X_\theta$ modulo $w_i$ is split, and
$X_{\theta}$ contains a $k({\theta})_{w_i}$-point by Lemma \ref{LangWeil}(b).

The following final remark will be used in the next \emph{Step}.
For all places $w\in\Omega_{k(\theta)}$ outside $S_2,$ the existence of $k(\theta)_w$-points of $X_\theta$ is deduced as above
by applying Lemma \ref{LangWeil}. In other words, these local points are in fact integral (with respect to $w$) points
obtained by lifting of points of finite residue fields.
As remarked in the proof of Lemma \ref{LangWeil}, they can be chosen to be integral points of the integral
model $\mathcal{X}\setminus(\bigsqcup_r\mathcal{X}_r\setminus\mathcal{Y}_r)$ of the Zariski open
$X\setminus(\bigsqcup_rX_r\setminus Y_r)\subset X$ where $r$ runs through $\{1,\ldots,l,\infty\}.$
\end{step}

\medskip
\begin{step}\label{orthogonal to Br} \emph{Orthogonality to the Brauer group.}

For each $w\in\Omega_{k(\theta)},$ let $M_w$ denote the $k(\theta)_w$-point we found
on the fibre $X_\theta$ in \emph{Step} \ref{local points}. We know by continuity of the
Brauer\textendash Manin pairing and the projection formula that
$$\sum_{w\in S_2\otimes_kk(\theta)}\inv_w(A(M_w))=\sum_{v\in S_2}\inv_v(\langle A,\tau_v\rangle_v)=0$$
for all $A\in\Lambda\cup E.$
In this \emph{Step}, we are going to compute
$\inv_w(A(M_w))$ for $w$ outside $S_2$ and modify certain $M_w$ such that they satisfy the equality
$$\sum_{w\in\Omega_{k(\theta)}}\inv_w(A(M_w))=0$$
for all $A\in\Lambda.$

\begin{substep}\label{classification of places}\emph{Classification of places of $k(\theta).$}

As a 0-cycle $\theta$ is sufficiently close to $x_r+x'_r$ for $1\leqslant r\leqslant l,$ this means that there exists
a place $w_r^0$ of $k(\theta)$ above $v_r$ such that the extensions $k(\theta)_{w^0_r}/k_{v_r}$ and
$k(\theta)(w^0_r)/k(v_r)$ are trivial and the image of $\theta$ in $\mathbb{A}^1(k(\theta)_{w^0_r})$
is sufficiently close to $x_r.$
Hence $w^0_r(Q_r(\theta))=v_r(Q_r(\theta))=v_r(Q_r(x_r))=1,$ \emph{a fortiori} $w^0_r(\theta)\geqslant0$ since
the coefficients of $Q_r$ are integers at $v_r.$ Similarly, there exists a place $w^0_\infty$ of $k(\theta)$
above $v_\infty$ such that $w^0_\infty(1/\theta)=1.$

We consider the reduction of $\theta\in V_0\subset\mathbb{P}^1$ modulo a place $w\in\Omega_{k(\theta)}\setminus S_2\otimes_kk(\theta),$
there are three possibilities:
\begin{itemize}
\item[(a)] it is in $\mathcal{V}_0,$ if and only if $w(Q_r(\theta))=0$ (\emph{a fortiori} $w(\theta)\geqslant0$), we denote by
            $\Omega_0\subset\Omega_{k(\theta)}$ the subset of such places;
\item[(b)] it is in one (unique) of the $\widetilde{m}_r\mbox{'s }(1\leqslant r\leqslant l),$ if and only if $w(Q_r(\theta))>0$
            (\emph{a fortiori} $w(\theta)\geqslant0$), we denote by $\Omega_r\subset\Omega_{k(\theta)}$ the subset of such
            places, in particular $w^0_r\in\Omega_r;$
\item[(c)] it is in $\widetilde{\infty}$ if and only if $w(\theta)<0,$ we denote by $\Omega_\infty\subset\Omega_{k(\theta)}$
            the subset of such places. Then $\Omega_\infty\subset I\otimes_kk(\theta)$ since $\theta$ is an $S_2\cup I$-integer.
\end{itemize}
The subsets $\Omega_0,\Omega_1,\ldots,\Omega_l,\Omega_\infty$ form a partition of $\Omega_{k(\theta)}\setminus S_2\otimes_kk(\theta),$
and $\Omega_r$ is finite for $r\in\{1,\ldots,l,\infty\}.$
\end{substep}

\begin{substep}\label{computation of A(M_w)} \emph{Computation of $A(M_w).$}

For $w\in\Omega_{k(\theta)}\setminus S_2\otimes_kk(\theta),$ we want to compute $A(M_w)$ for $A\in\Lambda\cup E\subset \Br(Y_0).$
By construction in \emph{Step} \ref{local points}, the point $M_w$ is a lifting (by Hensel's lemma) of its modulo-$w$-reduction
$M(w)$ of $\mathcal{Y}_r$ according to $w\in\Omega_r(r\in\{0,1,\ldots,l,\infty\}).$

We find that, by \cite[Cor. 2.4.3]{Harari}, for $A\in \Lambda\cup E$
\begin{itemize}
\item[(a)] $\inv_w(A(M_w))=0$ if $w\in\Omega_0;$
\item[(b)] $\inv_w(A(M_w))=w(Q_r(\theta))\cdot\partial_{A,r,M(w)}(F_{r,M(w)})$ if $w\in\Omega_r(1\leqslant r\leqslant l),$
            where $\partial_{A,r,M(w)}$ is the evaluation of the element
            $\partial_{A,r}\in H_{\scriptsize{\mbox{\'et}}}^1(\mathcal{Y}_r,\mathbb{Q}/\mathbb{Z})$
            at the point $M(w)$ of $\mathcal{Y}_r,$ and where $F_{r,M(w)}\in \Gal(\mathcal{W}_r/\mathcal{Y}_r)$ is
            the Frobenius element at $M(w).$
\item[(c)] $\inv_w(A(M_w))=w(1/\theta)\cdot\partial_{A,\infty,M(w)}(F_{\infty,M(w)})$ if $w\in\Omega_\infty,$ where
            $\partial_{A,\infty,M(w)}$ is the evaluation of
            $\partial_{A,\infty}\in H_{\scriptsize{\mbox{\'et}}}^1(\mathcal{Y}_\infty,\mathbb{Q}/\mathbb{Z})$ at
            $M(w)\in\mathcal{Y}_\infty,$ and where $F_{\infty,M(w)}\in \Gal(\mathcal{W}_\infty/\mathcal{Y}_\infty)$ is
            the Frobenius element at $M(w).$
\end{itemize}
\end{substep}

\begin{substep}\label{computation for w infty} \emph{Computation for $w\in\Omega_\infty.$}

If $A\in E,$ it comes from an element of $\Br(k(t))$ which becomes $0$ in
$\Br(k'(t)).$ Moreover, the place $v$ of $k$ below the place $w\in\Omega_\infty$ belongs to $I\otimes_{k}k(\theta),$
thus the field $k(\theta)_w$ contains $k_v=k'_{v'}$ for all $v'$ above $v,$
then $A=0$ in $\Br(k(\theta)_w(t)).$
Hence $\inv_w(A(M_w))=0$ for $w\in\Omega_\infty$ and $A\in E.$

Since $v$ splits completely also in $k'_{\infty},$ the Frobenius element $F_{\infty,M(w)}$ lies automatically in $\Gal(K'_\infty/K_\infty k'_\infty).$
\end{substep}

\begin{substep}\label{computation for w r} \emph{Computation for $w\in\Omega_r(1\leqslant r\leqslant l).$}

\emph{A priori}, the element $\sum_{w\in\Omega_r}w(Q_r(\theta))\cdot F_{r,M(w)}$ (additive notation) lies in
the abelian group $\Gal(K'_r/K_r),$ we are going to prove that this element belongs to the subgroup
$\Gal(K'_r/K_rk'_r).$

Let $\rho_r$ be an arbitrary element of $G'_r=H^1(\Gal(K_rk'_r/K_r),\mathbb{Q}/\mathbb{Z})\subset G_r,$
then there exists an element $A_r$ of $E\subset \Br(k(t))$ whose only possible non-zero residue is at $m_r$
giving value $\rho_r.$ In this case, we have $\inv_w(A_r(M_w))=0$ for all $w\in\Omega_{r'}$ if $r'\neq r$ and
$1\leqslant r'\leqslant l.$ On the other hand, as the specialization of $A_r\in \Br(k(X))$ at $X_\theta$
comes from $\Br(k(\theta)),$ we find the equality $\sum_{w\in\Omega_{k(\theta)}}\inv_w(A_r(M_w))=0.$
Combining the computations for places in $\Omega_0$ and $\Omega_\infty$ done in \emph{Steps} \ref{computation of A(M_w)} and
\ref{computation for w infty}
$$\sum_{w\in\Omega_r}\inv_w(A_r(M_w))=0,$$
in other words,
$$\sum_{w\in\Omega_r}\rho_r\left(w(Q_r(\theta))\cdot F_{r,M(w)}\right)=0$$
for all $\rho_r\in G'_r.$
Hence the element $\sum_{w\in\Omega_r}w(Q_r(\theta))\cdot F_{r,M(w)}$ lies in $\Gal(K'_r/K_rk'_r).$
\end{substep}

\begin{substep}\label{modification} \emph{Modification of $M_{w_r^0}(1\leqslant r\leqslant l,\mbox{ }r=\infty)$.}

Recall that for $1\leqslant r\leqslant l$ there exists a place $w^0_r\in\Omega_r$ above $v_r$ with
$w^0_r(Q_r(\theta))=1$ and $w^0_\infty(1/\theta)=1$ for $r=\infty$, moreover the extensions $k(\theta)_{w^0_r}/k_{v_r}$ and $k(\theta)(w^0_r)/k(v_r)$
are trivial. As $v_r$ splits completely in $k'_r,$ the Frobenius element $F_{r,M(w^0_r)}$
lies in $\Gal(K'_r/K_rk'_r).$

The geometric Chebotarev's density theorem \cite[Lem. 1.2]{Ekedahl} implies the existence of a
point $M'(w^0_r)$ of $\mathcal{Y}_r$ whose Frobenius element $F_{r,M'(w^0_r)}$ is exactly
the element $F_{r,M(w^0_r)}-\sum_{w\in\Omega_r}w(Q_r(\theta))\cdot F_{r,M(w)}\in \Gal(K'_r/K_rk'_r).$
We lift it to a $k(\theta)_{w^0_r}$-point $M'_{w^0_r}$ of $X_\theta.$ We find a point $M'_{w^0_\infty}$ with similar property for $r=\infty$ as well.

For each $1\leqslant r\leqslant l\mbox{ and }r=\infty,$ we replace $M_{w_r^0}$ by $M'_{w^0_r}$ and we keep $M_w$ for all
$w\in\Omega_r\setminus\{w^0_r\}.$ We verify that $$\sum_{w\in\Omega_r}\inv_w(A(M_w))=0$$
for $A\in\Lambda$ and for all $1\leqslant r\leqslant l.$ Therefore
$$\sum_{w\in\Omega_{k(\theta)}}\inv_w(A(M_w))=0$$ for all $A\in\Lambda.$
\end{substep}
\end{step}

\medskip
\begin{step}\label{end of proof} \emph{End of the proof.}

The specialization map $sp_\theta$ maps $\Lambda$ on to $\Br(X_\theta)/\Br(k(\theta)),$ the family
$\{M_w\}_{w\in\Omega_{k(\theta)}}$ of local points on $X_\theta$ is then orthogonal to $\Br(X_\theta).$
The hypothesis gives a global 0-cycle $z'$ of degree $1$ on $X_\theta$ such that
$z'$ and $M_w$ have the same image in $\CH_0(X_{\theta,w})/N$ for all $w\in S_2\otimes_kk(\theta).$
The 0-cycle $z'$ is regarded as a 0-cycle of degree $d\equiv\delta(\mbox{mod }ad_0)$ on $X,$
by subtracting a suitable multiple of $az_0$ from $z',$ we get a global 0-cycle $z$ of degree $\delta$ on $X$
such that $z$ and $z_v$ have the same image in $\CH_0(X_v)/N$ for all $v\in S.$
\end{step}
\end{proof}

\begin{rem}\label{weakBr}
Actually a sharper statement can be proved: we can allow one fibre over a certain rational point of $\mathbb{P}^{1}$  not to be abelian-split and only require this fibre to contain an irreducible component of multiplicity $1.$
In fact, we can choose the point $\infty$ to be the image of this fibre. We fix a multiplicity $1$ irreducible component $X^{irr}_{\infty}$  of the fibre $X_{\infty},$ then we set $K_{\infty}$ to be its function field and choose a non-empty smooth open subset $Y_{\infty}$ of $X^{irr}_{\infty}$ which does not intersect with other irreducible components of $X_{\infty}$. In the above proof, only the fibres over $\widetilde{\infty}$ have different properties.
To see that the argument still works, we need to check that the fibre $X_{\theta}$ has local points at all places $w\in\Omega_{k(\theta)}$.
The only difference appears when the reduction of $X_{\theta}$ modulo $w\notin S_{2}\otimes_{k}k(\theta)$ is a fibre over a closed point of $\widetilde{\infty}.$ In such a case  $w\in\Omega_{\infty}\subset I\otimes_{k}k(\theta)$,  then a point of its reduction lifts to a $k(\theta)_{w}$-point of $X_{\theta}$ according to Lemma \ref{LangWeil}(d).
All the computation still works. Furthermore, we use only the fact that $\Lambda\subset \Br(Y\setminus X_{\infty}),$ therefore it suffices to assume a weaker form of the hypothesis $(\Br)$: the subjectivity of $\Br(Y\setminus X_{\infty})\to \Br(X_{\eta})/\Br(k(t)).$

Let $Y'$ be an open subset of $Y$ such that their projections to $\P^1$ coincide.
In the hypothesis and in the proof one can even replace $Y$ by $Y'$, the whole argument still works. Therefore the hypothesis
$(\Br)$ can even be  weaken to the subjectivity of $\Br(Y'\setminus X_{\infty})\to \Br(X_{\eta})/\Br(k(t)).$ In other words, this weakened hypothesis only requires that on each non-split fibre other than $\infty$ there exists at least one irreducible component such that the elements in $\Lambda$ are all unramified with respect to the residue map at the generic point of that component.

\end{rem}


\section{Explicit examples}\label{application}

First of all, in \S \ref{Wei}, we explain why Theorem \ref{main_thm} can be applied to varieties considered by Wei \cite[Thm. 4.1]{Wei}
to obtain the exactness of $(\E)$. In \S \ref{previous}, we look at two extreme cases, in which the hypothesis $(\Br)$ is automatically
satisfied and Theorem \ref{main_thm} reduces into two known results. For each case, we give an example.

\subsection{The recent result of Wei on normic equations}\label{Wei}

In a recent paper of Wei \cite{Wei}, by explicit computations of elements in the Brauer group he has proved
that the Brauer\textendash Manin obstruction is the only obstruction to the Hasse principle for 0-cycles of degree $1$ on
varieties defined by several specific normic equations, which can not be covered by previous work. One of them is presented as follows.

Consider the equation over a number field $k$
$$N_{K/k}(\overrightarrow{\textbf{x}})=P(t)$$
where $K/k$ is a finite abelian extension of degree $n$ and $P(t)\in k[t]$ is a non-zero polynomial.
Let $X^{\tiny{\textup{CTHS}}}$ be the CTHS partial compactification (named after the authors) of the closed
sub-variety of $\mathbb{A}^{n+1}$ defined by the equation, it admits a morphism $X^{\tiny{\textup{CTHS}}}\to\mathbb{P}^1,$
\emph{cf.} \cite[\S 2]{CTHS} for the construction.
Let $X$ be a smooth compactification of $X^{\tiny{\textup{CTHS}}},$ the extended fibration $X\to\mathbb{P}^1$ satisfies (gen) and (ab-sp).
We denote by $T$ the norm one $k$-torus $R^1_{K/k}\mathbb{G}_m$ defined by $N_{K/k}(\overrightarrow{\textbf{x}})=1.$
Assuming an additional hypothesis that $\sha^2_\omega(\widehat{T})_P=\sha^2_\omega(\widehat{T}),$
\emph{cf.} \cite[\S 2]{CTHS} for definition of these groups, Wei proved that the Brauer\textendash Manin obstruction is the only obstruction
to the Hasse principle for 0-cycles of degree $1$ on $X,$ \cite[Thm. 4.1]{Wei}. He gave an explicit example
satisfying this additional hypothesis, \cite[Cor. 3.3]{Wei} or Corollary \ref{maincor} of the present paper. His example is essentially new in the sense that
$\Gal(K/k)\simeq\mathbb{Z}/n\mathbb{Z}\times\mathbb{Z}/n\mathbb{Z}$ is not cyclic. It is remarked that the additional
hypothesis is equivalent to the surjectivity of $\Br(X^{\tiny{\textup{CTHS}}})\to \Br(X_\eta)/\Br(k(t)),$
\emph{cf.} the remark before Corollary 3.4 in \cite{Wei}.
It is pointed out by Smeets that Wei's hypothesis is equivalent to the weakened hypothesis $(\Br)$ in Remark \ref{weakBr}, \textit{cf.} \cite[Prop. 4.1]{Smeets}.
Therefore we obtain the exactness of $(\E)$ for such varieties (Corollary \ref{maincor}), we also recover Wei's result on the
Hasse principle for 0-cycles of degree $1$.

\subsection{Previous results for 0-cycles}\label{previous}

\subsubsection{} If all the closed fibres of $\pi:X\to\mathbb{P}^1$ are split (in particular, if they are geometrically integral),
the hypotheses (ab-sp) and $(\Br)$ are automatically verified. This special case has been proved by the author
in \cite{Liang2}, but the following example has not been discussed.

Consider normic equations $N_{K/k}(\overrightarrow{\textbf{x}})=P_{s}(t)$ where $K/k$ is a finite extension (not necessarily abelian nor Galois) and where $P_{s}(t)\in k[s,t]$ is a polynomial viewed as a family of polynomials parameterized via $s.$ It defines a fibration $X\to\mathbb{P}^{1}$.  If for arbitrary $\bar{k}$ value $\theta$ of $s$, the polynomial $P_{\theta}(t)\in k(\theta)[t]$ is not zero, then the equation defines a geometrically integral $k(\theta)$-variety $X_{\theta}$.
If $P_{s}(t)=c(t-a(s))^{m}(t-b(s))^{n}$ with $m,n\in\mathbb{N},$ $c\in k^{*},$ and $a(s),b(s)\in k[s],$ according to \cite[Thm. 1]{SJ} the Brauer\textendash Manin obstruction is the only obstruction to weak approximation for rational points on $X_{\theta}.$ Theorem \ref{main_thm} implies that $(\E)$ is exact for $X$. Similar examples can be constructed  considering other normic equations.

\subsubsection{}
The hypothesis $(\Br)$ is also automatically verified if the morphism $\Br(k(t))\to \Br(X_\eta)$ is surjective
(\emph{e.g.} certain fibrations in Ch\^atelet surfaces).
By Proposition \ref{comparison Br}
the Brauer group $\Br(X_\theta)$ consists only elements coming from $\Br(k(\theta))$
for all $\theta\in \textsf{Hil}.$ The condition (1) (resp. (2)) of the main theorem \ref{main_thm}
becomes
\begin{itemize}
\item[-]
The $k(\theta)$-variety $X_\theta$ satisfies the Hasse principle (resp. weak approximation)
for rational points or for 0-cycles of degree $1.$
\end{itemize}
This special case has also been proved by the author in \cite{Liang3}.
The following is an example of this kind, nevertheless, it has not appeared in the literature.

Consider the normic equation $N_{K/k}(\overrightarrow{\textbf{x}})=P_{s}(t)$, where $K/k$ is an abelian extension of degree $4$ and where $P_{s}(t)=a(s)t^{2}+b(s)t+c(s)\in k(s)[t]$ is a quadratic polynomial with $a(s),b(s),c(s)\in k(s)$ such that
$\Delta(s)=b(s)^{2}-4a(s)c(s)\in K(s)^{*2}\setminus k(s)^{*2}.$
It defines birationally a fibration $X$ over $\mathbb{P}^{1}$ via the parameter $s$. It can be shown by the same argument as \cite[Lem. 3.2]{Liang5} that all such fibrations satisfy (ab-sp).
The quadratic polynomial $P_{s}(t)$ defines a double cover of the projective line and hence defines a generalized Hilbertian subset $\textsf{Hil}\subset\mathbb{P}^{1},$ by shrinking  $\textsf{Hil}$ one may assume that the residual field $k(\theta)$ is linearly disjoint from $K/k$ for all $\theta\in\textsf{Hil}$.
By definition $P_{s}(t)$ is irreducible over $k(s)$ and splits over $K(s),$ hence for any $\theta\in\textsf{Hil}$ the specialization $P_{\theta}(t)$ is an irreducible polynomial over $k(\theta)$ that splits over the degree $4$ extension $K(\theta).$ The fibre over such a point $\theta$ satisfies weak approximation for rational points according to \cite[Thm. 1]{DSW}. We also know that the Brauer group of the generic fibre of such a fibration is trivial \cite[Thm. 1]{DSW2}. Theorem \ref{main_thm} applies and we obtain the exactness of $(\E)$ for $X$.



\bibliographystyle{alpha}
\bibliography{mybib1}
\end{document}